\documentclass[12pt,reqno]{amsart} \usepackage{setspace}
\usepackage{amsmath, amssymb, amsthm}
\numberwithin{equation}{section}
\usepackage{amsmath}
\usepackage{amscd,amsthm,amsfonts,amsopn,amssymb,mathrsfs, latexsym}
\usepackage{epsfig,hyperref}
\theoremstyle{plain}

\newtheorem{lemma}{Lemma}
\theoremstyle{remark}

\DeclareMathOperator{\loc}{loc}
\allowdisplaybreaks
\def\be{\begin{equation}}
\def\ee{\end{equation}}
\parskip=\medskipamount

\def\ve{\varepsilon}
\def\vp{\varphi}

\def\arrowk{^\to{\kern -6pt\topsmash k}}
\def\arrowK{^{^\to}{\kern -9pt\topsmash K}}
\def\arrowr{^\to{\kern-6pt\topsmash r}}
\def\arrowvp{^\to{\kern -8pt\topsmash\vp}}
\def\arrowf{^{^\to}{\kern -8pt f}}
\def\arrowg{^{^\to}{\kern -8pt g}}
\def\arrowu{^{^\to}a{\kern-8pt u}}
\def\arrowt{^{^\to}{\kern -6pt t}}
\def\arrowe{^{^\to}{\kern -6pt e}}
\def\tk{\tilde{\kern 1 pt\topsmash k}}
\def\barm{\bar{\kern-.2pt\bar m}}
\def\barN{\bar{\kern-1pt\bar N}}
\def\barA{\, \bar{\kern-3pt \bar A}}

\begin{document}
\title{On the Hardy-Littlewood maximal function for the cube}

\author{Jean Bourgain}
\address{School of Mathematics, Institute for Advanced Study, 1
Einstein Drive, Princeton, NJ 08540.}
\email{bourgain\@math.ias.edu}
\thanks{The research was partially supported by NSF grants DMS-0808042 and
DMS-0835373.}
\begin{abstract}
It is sown that the Hardy-Littlewood maximal function associated to the cube in $\mathbb R^n$ obeys dimensional free bounds in $L^p$ for $p>1$.
Earlier work only covered the range $p>\frac 32$.
\end{abstract}

\maketitle
\centerline{(Dedicated to J.~Lindenstrauss)}

\section
{Introduction}
Let $B$ be a convex centrally symmetric body in $\mathbb R^n$ and define the corresponding maximal function
\be\label{1.1}
Mf(x) =M_Bf(x)=\sup_{t>0} \frac 1{|B|} \int_B|f(x+ty)|dy; f\in L^1_{\loc} (\mathbb R^n)
\ee
where $|B|$ denotes the volume of $B$.
For $1<p<\infty$, let $C_p(B)$ be the best constant in the inequality
\be\label{1.2}
\Vert M_Bf\Vert_p \leq C_p(B)\Vert f\Vert_p
\ee
while $C_1(B)$ is taken to satisfy the weak-type inequality
\be\label{1.3}
\Vert M_Bf\Vert_{1, \infty}\leq C_1(B)\Vert f\Vert_1.
\ee
Using the theory of spherical maximal functions, Stein \cite {St1} established the remarkable fact that for $B=B_2$ the Euclidean ball,
$C_p(B_2)$ may be bounded independently of the dimension $n$, for all $p>1$.
The author obtained the boundedness of $C_2(B)$ by an absolute constant, independently of $B$ as above (cf \cite {B1}) and this statement was generalized to
$C_p(B)$, $p>\frac 32$ in \cite{B2} and \cite{C}.
On the other hand, it is shown in \cite{S-S} that $C_1(B)\lesssim n\log n$ (see also \cite{N-T}).
Note that the constants $C_p(B)$ are invariant under linear transformation, i.e. $C_p(B)=C_p\big(u(B)\big)$ for $u\in CL_n(\mathbb R)$.
It is convenient to choose $u$ as to make $B$ isotropic, meaning that
\be\label{1.4}
\int_B|\langle x, \xi\rangle|^2 dx =L(B)^2 \text { for all unit vectors $\xi\in\mathbb R^n$.}
\ee

If $B$ is in isotropic position, then all $(n-1)$-dimensional central sections of $B$ are approximately of the same volume $\sim L(B)^{-1}$, up to absolute
constants, se \cite{B1} for details.
Recall at this point that $L(B)$  is known to be bounded from below by an absolute constant.
Conversely, the uniform bound from above is a well-known open problem with several equivalent formulations.
While such bounds were obtained for various classes of convex symmetric bodies (in particular zonoids), the best currently available general estimate on
$L(B)$ is $O(n^{1/4})$.
Interestingly, this issue did not impact the proofs of the dimension free bounds obtained in \cite {B1}, \cite{B2}, \cite{C}.
Next, following \cite{M}, denote by $Q(B)$ the maximum volume of an orthogonal $(n-1)$-dimensional projection of the isotropic position $S(B)$ of $B$.
That is
\be\label{1.5}
Q(B) =\max_\xi |\pi_\xi\big(S(B)\big)|
\ee
denoting $\pi_\xi$ the orthogonal projection on $\xi^\bot$, $\xi\in\mathbb R^n$ a unit vector.
It is proven in \cite{M} that for all $p>1$, one may estimate $C_p(B)$ in terms of $L(B)$ and $Q(B)$.
Consequently, \cite{M} obtains dimension free maximal bounds in the full range $p>1$ for $B=B_q=$ the unit ball with respect to the $\ell^q$-norm in $\mathbb
R^n$, provided $1\leq q < \infty$.
For $q=\infty$, one gets
\be\label{1.6}
Q(B_\infty) =\sqrt n
\ee
resulting in no further progress for the Hardy-Littlewood maximal function for the cube.

Still for the cube, a brakethrough was made recently in the works of Aldaz \cite {Al} and Aubrun \cite{Au}, that disprove a dimension free weak $(1, 1)$
maximal inequality for $B_\infty$.
More specifically, it is shown in \cite{Au} that
\be\label{1.7} 
C_1(B_\infty^{(n)}) > c(\ve)(\log n)^{1-\ve} \text { for all $\ve>0$}.
\ee
The purpose of this work is to prove that on the other hand

\noindent
{\bf Theorem.} \ {\sl
$C_p(B_\infty)<C_p$ for all $p>1$.
}

While it is reasonable to believe that this statement holds in general, our argument is based on a very explicit analysis which does not immediately carry
over to other convex symmetric bodies.
But the results of \cite{Al} and \cite{Au} are certainly inviting to a further study of $M_{B_\infty}$ which after all, together with $M_{B_2}$, is the most
natural setting of the Hardy-Littlewood maximal operator.

Let us next give a brief description of our approach.

Denote in the sequel $B=B_\infty^{(n)}= \big[-\frac 12, \frac 12\big]^n\subset\mathbb R^n$ with Fourier transform
\be\label{1.8}
m(\xi)=\hat 1_B(\xi) =\prod^n_{j=1} \frac{\sin \pi\xi_j}{\pi\xi_j}.
\ee
Then
\be
\label{1.9}
\int_Bf(x+ty)dy =\int_{\mathbb R^n} \hat f(\xi) m(t\xi) e(x.\xi)d\xi
\ee
(with notation $e(y) =e^{2\pi iy})$ which reduces matters to the study of the Fourier multiplier $m(\xi)$.
It satisfies in particular the estimates
\be\label{1.10}
|m(\xi)|<\frac C{|\xi|} \text { for } |\xi|\to\infty
\ee
and
\be \label{1.11}
|\langle\nabla m(\xi), \xi\rangle|<C.
\ee

In fact, \eqref{1.10}, \eqref{1.11} hold in general for isotropic convex symmetric $B$, $|B|=1$, replacing \eqref{1.10} by
\be\label{1.12}
|\hat 1_B(\xi)|< \frac C{L(B)|\xi|}.
\ee
See\cite{B1}.

The estimates \eqref{1.10}, \eqref{1.11} set the limitation $p>\frac 32$ in bounding  $\Vert M_B\Vert_p$. Following \cite {B2}
for instance, a faster decay in \eqref{1.10} would allow to reach smaller values of $p$.
Now a quick inspection of \eqref{1.8} shows, roughly speaking, that most of the time $m(\xi)$ decays much faster and the worst case scenario \eqref{1.10} only
occurs for $\xi$ confined to narrow conical regions along the coordinate axes.
Thus our strategy will consist in making suitable localizations in fourier space which contributions will be treated using different arguments.

For $\Omega\in L^1(\mathbb R^n)$, denote for $t>0$ the scaling $\Omega_t(x)=\frac 1{t^n}\Omega\big(\frac xt\big)$ satisfying
$\hat\Omega_t(\xi)=\hat\Omega(t\xi)$.
Denote $H$ the Gaussian distribution on $\mathbb R^n$, $\hat H(\xi)= e^{-|\xi|^2}$.
We make a decomposition
\be\label{1.13}
1_B=(1_B*H)+\sum^\infty_{s=1} \Omega^{(s)} \ \text { with } \,   \Omega^{(s)} =1_B* H_{2^{-s}} -1_B*H_{2^{-s+1}}
\ee
and consider the maximal function associated to each $\Omega^{(s)}$.

Recall the following simple $L^2$-estimate (Lemma 3 in \cite{B1}).

\begin{lemma}\label{Lemma1}
Consider a kernel $K\in L^1(\mathbb R^n)$ and introduce the quantities
$$
\alpha_j=\max_{|\xi|\sim 2^j} |\hat K(\xi)|\text { and } \beta_j= \max_{|\xi|\sim 2^j}|\langle \nabla\hat K(\xi), \xi\rangle|\quad (j\in\mathbb Z).
$$
Then
\be\label{1.14}
\Vert\sup_{t>0} |f* K_t| \ \Vert_2 \leq C\Gamma(K)\Vert f\Vert_2 
\ee 
with
\be\label{1.15}
\Gamma(K)=\sum_{j\in\mathbb Z} \alpha_j^{1/2} (\alpha_j+\beta_j)^{1/2}.
\ee
\end{lemma}

Since
$$
\big|\widehat{\Omega^{(s)}} (\xi)|=|m(\xi)| \, \big| e^{-4^{-s}|\xi|^2}- e^{-4^{-s+1}|\xi|^2}\big|
$$
it follows that
\be\label{1.16}
\Vert\sup_{t>0} |f*(\Omega^{(s)})_t  | \ \Vert_2 \leq C2^{-s/2}\Vert f\Vert_2.
\ee
Taking $1<p<2$, our aim is to interpolate \eqref{1.16} with an estimate
$$
\Vert\sup_{t>0} |f* (\Omega^{(s)})_t| \, \Vert_p \leq A(p, s) \Vert f\Vert_p
$$
or
\be\label{1.17}
\Vert\sup_{t>0}|f*(1_B*H_{2^{-s}})_t| \, \Vert_p \leq A(p, s)\Vert f\Vert_p.
\ee
In order to establish \eqref{1.17}, we follow the approach in \cite{M}.

\begin{lemma}\label{Lemma2}
Assume the multiplier operator with multiplier
\be\label{1.18}
|\xi|m(\xi) e^{-4^{-s}|\xi|^2}
\ee
acts on $L^p(\mathbb R^n)$ with operator norm bounded by $A(p, s)$.
Then \eqref{1.17} holds with a proportional constant.
\end{lemma}

The statement follows from the argument in \cite {M}, based on analytic interpolation and a suitable admissible family of Fourier multiplier operators.

In the present situation, rather than taking $K=1_B$ in (9) of \cite {M}, we let $K=1_B*H_{2^{-s}}$.
It is important to note that in the crucial Lemma \ref{Lemma2} from \cite{M}, only the bound on $L(B)$ is required but not on $Q(B)$ (which enters at a later stage).
In fact, the essential input in \cite{M}, Lemma \ref{Lemma2} are bounds on
$$
\sup_{|\xi|=1} \int_B|\langle x, \xi\rangle|^k dx
$$
for fixed $k\geq 1$.
In our setting, we obtain
\be\label{1.19}
\int|\langle x, \xi\rangle|^k (1_{B_\infty}*H_{2^{-s}})(x)dx
\ee
which is easily evaluated.
Indeed, since the distribution $1_{B_\infty}*H_{2^{-s}}$ is symmetric in each coordinate $x_i$, application of Khintchine's inequality implies 
for $|\xi|=1$
$$
\begin{aligned}
(1.19) &< C_k\int\Big(\sum^n_1 x_i^2 \xi_i^2\Big)^{\frac k2} (1_B* H_{2^{-s}})(x)dx\\
&\leq C_k\iint \Big[\Big(\sum(x_i-y_i)^2 \xi_i^2\Big)^{\frac 12}+|\xi|\Big]^k 1_B(y)H_{2^{-s}}(x-y)dxdy\\
& <C_k+C_k\int\Big(\sum x_i^2\xi_i^2\Big)^{\frac k2} H_{2^{-s}}(x)dx\\
&=C_k+C_k 2^{-sk}\int\Big(\sum x_i^2\xi_i^2\Big)^{\frac k2} H(x)dx\\
&<C_k
\end{aligned}
\eqno{(1.20)}
$$
for $s\geq 0$.

Returning to \eqref{1.18}, we proceed further as in \cite{M}, writing
$$
|\xi|m(\xi)e^{-4^{-s}|\xi|^2}\hat f(\xi) =\sum^n_{i=1} \widehat{R_if}(\xi)\widehat{\mu_i} (\xi)
\eqno{(1.21)}
$$
with $\mathcal R_i$ the $i^{th}$ Riesz transform and $\mu_i=\partial_{x_i} (1_B*H_{2^{-s}})$.

Arguing by duality, take $g\in L^{p'}(\mathbb R^n)$, $\frac 1p+\frac 1{p'}=1$, $\Vert g\Vert_{p'}\leq 1$ and
estimate
$$
\begin{aligned}
\int|\xi|m(\xi) e^{-4^{-s}|\xi|^2} \hat f(\xi)\hat g(\xi) d\xi&= \sum^n_{i=1} \langle R_i f, g*u_i\rangle\\
&\leq \Big\Vert\Big(\sum |R_i f|^2)^{\frac 12}\Vert_p. \Big\Vert\Big(\sum|g*\mu_i|^2\Big)^{\frac 12}\Big\Vert_{p'}.
\end{aligned}\eqno{(1.22)}
$$
For the first factor in (1.22), use Stein's dimensional free bound on the Riesz transform (see \cite{St2})
$$
\Big\Vert \Big(\sum|R_i f|^2\Big)^{\frac 12}\Big\Vert_p\leq A_p \Vert f\Vert_p \ \text { for } \ 1<p<\infty.\eqno{(1.23)}
$$
This reduces the issue to an estimate on
$$
\Big\Vert\Big( \sum |g*\mu_i|^2\Big)^{\frac 12}\Big\Vert_p\eqno{(1.24)}
$$
for $ 2\leq p<\infty$.

Note that since $\big(\sum_i |\widehat{\mu_i}(\xi)|^2\big)^{\frac 12} \leq |\xi| \, |m(\xi')|< C $  by (1.10),
$$
\Big\Vert\Big(\sum |g*\mu_i|^2\Big)^{\frac 12} \Big\Vert_2 \leq C\Vert g\Vert_2.\eqno{(1.25)}
$$
Bounding (1.24) for $p=\infty$ amounts to an estimate on
$$
\sup_{|\eta|=1} \Big\Vert\sum^n_{i=1} \eta_i\mu_i\Big\Vert_1 =\Vert\nabla_\eta (1_B*H_{2^{-s}})\Vert_1.\eqno{(1.26)}
$$
Clearly
$$
(1.26) \leq \Vert\nabla_\eta (H_{2^{-s}})\Vert_1 =2^s\Vert\langle\nabla H, \eta\rangle\Vert_1 \lesssim 2^s\eqno{(1.27)}
$$
implying
$$
\Big\Vert\Big(\sum |g*\mu_i|^2\Big)^{\frac 12} \Big\Vert_\infty \leq C2^s\Vert g\Vert_\infty
$$
and
$$
\Big\Vert\Big(\sum|g*\mu_i|^2\Big)^{\frac 12} \Big\Vert_p \leq C2^{s(1-\frac 2p)}\Vert g\Vert_p \ \text { for } \ 2\leq p\leq\infty.\eqno{(1.28)}
$$
In particular, (1.17) holds with $A(p, s)< C_p2^s$ and interpolation with (1.16) is conclusive for $p>\frac 32$.

In order to prove the Theorem, it will suffice to establish an inequality of the form

\begin{lemma}\label{Lemma3}
For $R>1$ and $\mu_i=\partial_i(1_{B_\infty}* H_{\frac 1R})$, there is an inequality
$$
\Big\Vert\Big(\sum^n_{i=1} |f*\mu_i|^2\Big)^{\frac 12} \Big\Vert_p\leq C_p(\ve)R^\ve\Vert f\Vert_p\eqno{(1.29)}
$$
for all $2\leq p<\infty$ and $\ve<0$.
\end{lemma}

The proof of Lemma \ref{Lemma3} will occupy the remainder of the paper.
As mentioned before, the explicit form of $m(\xi)$ is essential in the argument.
In the next section, we introduce a new collection of Fourier multiplier operators that will enable to perform certain localizations in Fourier space.
The proof of (1.29) will then proceed by analyzing the expression $\big(\sum_i |f*\mu_i|^2\big)^{\frac 12}$ on each of these regions.

\section{Localization in Fourier space}

The following statement is a particular instance of Pisier's holomorphic semi-group theorem in $B$-convex spaces (\cite{P}).

\begin{lemma}\label{Lemma4}
Denote $\mathbb E_j$ a conditional expectation operator acting on the $j^{th}$ variable of $\mathbb R^n$.
Then, for $1<p<\infty$, the semi-group
\be \label{2.1}
S_t =\prod^n_{j=1} \big(\mathbb E_j+e^{-t} (1-\mathbb E_j)\big) \quad (t\geq 0)
\ee
acting on $L^p(\mathbb R^n)$ admits a holomorphic extension.
Hence, for \hbox{$0\leq k\leq n$} the operator
\be\label{2.2}
\sum_{\substack{S\subset\{1, \ldots, n\}\\ |S|=k}} \Big(\prod_{j\not\in S} \mathbb E_j\Big) \prod_{j\in S} (1-\mathbb E_j)
\ee
acts on $L^p(\mathbb R^n)$ with norm bounded by $C_p^k$.
\end{lemma}

We may replace the expectation operators $\mathbb E_j$ by convolution operators using a standard averaging procedure over translations.
Let $\mathbb E$ be the expectation operator corresponding to the partition of $\mathbb R$ in the intervals $[k, k+1[$, $k\in\mathbb Z$.
Thus its kernel is given by
$$
\Phi (x, y) =\sum_{k\in\mathbb Z} 1_{[k, k+1[} (x) 1_{[k, k+1[}(y).
$$
Averaging over translations, the operator $\int_0^1(\tau_\theta\mathbb E\tau_{-\theta}) d\theta$ is the convolution by $\eta =1_{[0, 1[}*1_{[0, 1[}$, i.e.
\be\label{2.3}
\eta(x) =(1-|x|)_+.
\ee

Lemma \ref{Lemma4} convexity therefore implies

\begin{lemma}\label{Lemma5}
Let $\eta$ be as in \eqref{2.3} and $(t_j)_{1\leq j\leq n}$ positive numbers.
Denote $T_j$ the convolution operator by $\eta_{t_j}$ in the $j$-variable.
Then, for \hbox{$0\leq k\leq n$}, the operator
\be\label{2.4}
\sum_{\substack {S\subset \{1, \ldots, n\}\\ |S|=k}} \ \prod_{j\not\in S} T_j \prod_{j\in S} (1-T_j)
\ee
acts on $L^p(\mathbb R^n)$, $1<p<\infty$, with norm bounded by $C_p^k$.
\end{lemma}

Returning to Lemma \ref{Lemma3}, set $t_j=t=R^{-\ve}$ for each $j=1, \ldots n$, with $\ve>0$ and a fixed small constant.
Denote $A_k$ the corresponding convolution operator \eqref{2.4},
which satisfies
\be\label{2.5}
\Vert A_k\Vert_p < C_p^k \ \text { for } \ 1<p<\infty.
\ee
Let $K=K(\ve, p)\in \mathbb Z_+$ and decompose $f$ as
\be\label{2.6}
f =(A_0+\cdots+ A_K) f+g.
\ee

Going back to the $L^2$-inequality (1.25), we obtain from Parseval that
\be\label{2.7}
\Big\Vert\Big[\sum^n_{i=1} |g*\mu_i|^2\Big]^{\frac 12}\Big\Vert_2\leq \rho \Vert f\Vert_2
\ee
with $\rho$ an upper bound on
\be\label{2.8}
\begin{aligned}
|m(\xi)| e^{-4^{-s}|\xi|^2} |1-\hat A_0(\xi)-\cdots-\hat A_K(\xi)|=\\
\prod^n_{j=1} \Big|\frac {\sin \pi\xi_j}{\pi\xi_j}\Big| e^{-4^{-s}|\xi|^2} \sum_{|S|>K} \, \prod_{j\not\in S} \hat\eta (t\xi_j) \prod_{j\in S}\big(1-\hat\eta(t\xi_j)\big).\\
&{}
\end{aligned}
\ee

\begin{lemma}\label{Lemma6}
For all $\delta>0$ and $k\geq 1$,
\be \label{2.9}
|m(\xi)|< C_k\Big(1+\sum_{|\xi_j|<R^\delta} \xi_j^2\Big)^{-\frac k2} R^{\delta k}
\ee
\end{lemma}

\begin{proof}

Denoting $I_0=\{ j=1, \ldots, n; |\xi_j|>1\}$, clearly
$$
\prod_{j\not\in I_0} \Big|\frac {\sin \pi\xi_j}{\pi\xi_j}\Big| < e^{-c\sum_{j\not\in I_0}\xi_j^2}
$$
while
$$
\prod_{j\in I_0} \Big|\frac {\sin \pi \xi_j}{\pi\xi_j}\Big|< e^{-c|I_0|}.
$$
Estimating
$$
\sum_{|\xi_j|<R^\delta} \xi_j^2 < R^{2\delta} |I_0|+\sum_{j\not\in I_0}\xi_j^2
$$
\eqref{2.9} follows.
\end{proof}

Returning to \eqref{2.8}, set $I_1=\{j=1, \ldots, n; |\xi_j|>R^{\frac \ve 5}\}$.
If $|I_1|>\frac K2$,
\be\label{2.10}
\eqref{2.8} \leq |m(\xi)|< R^{-\frac{\ve K}{10}}.
\ee
Assume $|I_1|\leq\frac K2$ and bound \eqref{2.8} by
\begin{alignat}{1}\label{2.11}
&|m(\xi)| \Big\{\sum_{\substack {S\cap I_1 =\phi\\ |S|\geq [\frac K2]}} \prod_{j\not\in S}\hat\eta (t\xi_j). \prod_{j\in S}
\big(1-\hat\eta (t\xi_j)\big)\Big\}=\notag \\
&|m(\xi)| \, \Big|\partial_r^{([\frac K2])} \Big[\prod_{\substack {1\leq j\leq n\\ j\not\in I_1}} \big(\hat \eta (t\xi_j)+r\big(
1-\hat\eta(t\xi_j)\big)\big)\Big]\Big|_{r=1}\Big|\leq\notag\\
&|m(\xi)| \, \Big[\sum_{j\not\in I_1} (1-\hat\eta(t\xi_j)\big) \Big]^{[\frac K2]}\lesssim\notag\\
&|m(\xi)| t^{2[\frac K2]}\Big[\sum_{j\not\in I_1} \xi_j^2\Big]^{[\frac K2]}\lesssim\notag\\
&R^{-2\ve[\frac K2]} R^{\frac 25 \ve[\frac K2]}\lesssim R^{-\ve[\frac K2]}.
\end{alignat}

Since $t=R^{-\ve}$, $ 1-\hat\eta(x)< cx^2$ for $|x| <1$ and \eqref{2.9}.

Combining \eqref{2.10}, \eqref{2.11}, it follows that we may take $\rho=R^{-\frac{\ve K}{10}}$ in \eqref{2.7}.

Since by (1.28), certainly 
\be\label{2.12}
\Big\Vert\Big(\sum^n_{i=1} |g*\mu_i|^2\Big)^{\frac 12}\Big\Vert_p \leq C_pR\Vert f\Vert_p \ \text { for } \ 1<p<\infty
\ee
interpolation between \eqref{2.7}, \eqref{2.12} implies that
\be\label{2.13}
\Big\Vert\Big(\sum^n_{i=1} |g*\mu_i|^2\Big)^{\frac 12}\Big\Vert_p\leq \Vert f\Vert_p
\ee
provided we choose $K=K(\ve, p)$ appropriately.

Thus we are left with estimating (1.24) for $g=A_kf$, $k\leq K$, that will be done using different arguments.

Write
\be\label{2.14}
\Big(\sum^n_{i=1} |A_k f*\mu_i|^2\Big)^{\frac 12}\leq \Big(\sum^n_{i=1}\Big|\sum_{\substack{|S|=k\\ i\not\in S}} 
\Gamma_S f*\mu_i\Big|^2\Big)^{\frac 12}
\ee
\be
\label{2.15}
\qquad\qquad +\Big(\sum^n_{i=1} \Big|\sum_{\substack{|S|=k\\ i\in S}} \Gamma_S f *\mu_i\Big|^2\Big)^{\frac 12}
\ee
denoting
$$
\Gamma_S=\prod_{j\in S} (1-T_j).\, \prod_{j\not\in S} T_j
$$
and $T_j$ the convolution by $\eta_t$ in $x_j$.

Any significant simplification is obtained by decoupling the variables in (2.14), (2.15).
We recall the procedure.

Let $(\gamma_i)_{1\leq i\leq n}$ be independent $\{0, 1\}$-valued random variables of mean $\frac 1k$ say and for $S\subset\{1, \ldots,
n\}$,
$|S|=k$ and $i\not\in S$, let
\be\label {2.16}
\sigma_{S, i} =\gamma_i \prod_{j\in S} (1-\gamma_j).
\ee
By construction
$$
\mathbb E_\omega[\sigma_{S, i}] =\frac 1k\Big(1-\frac 1k\Big)^k=c_k.
$$
Hence, by convexity
\begin{alignat}{1}\label{2.17}
\eqref{2.14}&\leq c_k^{-1} \mathbb E_\omega\Big[\Big(\sum^n_{i=1}\Big|\sum_{|S|=k, i\not\in S} \sigma_{S, i}(\omega) 
\big(\Gamma_S f*\mu_i\big)\Big|^2\Big)^{\frac 12} \Big]\notag\\
\Vert\eqref{2.14}\Vert_p& \leq c_k^{-1} \Big\Vert\Big(\sum^n_{i=1}\Big|\sum_{|S|=k, i\not\in S} \sigma_{S, i}(\omega)
\big(\Gamma_S f*\mu_i\big)\Big|^2\Big)^{\frac 12}\Big\Vert_p
\end{alignat}
for some $\omega$.
Denoting $I=\{ 1\leq i\leq n; \gamma_i(\omega)= 1\}$, \eqref{2.17} can be rewritten as
\be\label{2.18}
\Big\Vert\Big(\sum_{i\in I} \Big|\Big(\sum_{|S|=k, S\cap I=\phi} \Gamma_Sf\Big) *\mu_i\Big|^2\Big)^{\frac 12}\Big\Vert_p.
\ee
Let
\be \label{2.19}
F=\sum_{|S| =k, S\cap I=\phi} \, \prod_{j\in S} (1-T_j) \ \prod_{j\not\in I\cup S} T_j
\ee
which, applying Lemma \ref{Lemma5} in the variable $(x_j)_{j\not\in I}$, satisfies
\be\label{2.20}
\Vert F\Vert_p \leq C_p^k\Vert f\Vert_p
\ee 
and
\be\label{2.21}
\eqref{2.18} =\Big\Vert\Big(\sum_{i\in I}\Big|\Big(\prod_{i\in I} T_i\Big) F*\mu_i\Big|^2\Big)^{\frac 12}\Big\Vert_p.
\ee
Assuming we dispose of an estimate
\be\label{2.22}
\Big\Vert\Big(\sum_{i\in I} \Big|\Big(\prod_{i\in I} T_i\Big) g*\mu_i\Big|^2\Big)^{\frac 12} \Big\Vert_{L^p(\operatornamewithlimits\otimes\limits_{i\in I} dx_i)}
\leq b_0\Vert g\Vert_{L^p(\operatornamewithlimits\otimes\limits_{i\in I} dx_i)}
\ee
for $g\in L^p(\operatornamewithlimits\otimes\limits_{i\in I} dx_i)$, it will follow from \eqref{2.20} that \eqref{2.18} is bounded by 
$b_0C_p^k\Vert f\Vert _p$.

For \eqref{2.15} we proceed similarly, taking $S=\{i\}\cup S'$, $|S'|=k-1$.

Instead of \eqref{2.18}, we get
\be\label{2.23}
\Big\Vert\Big( \sum_{i\in I} \Big|\Big(\sum_{|S'|=k-1, S'\cap I=\phi} \Gamma_{\{i\}\cup S'} f\Big)*\mu_i\Big|^2\Big)^{\frac 12}\Big\Vert_p.
\ee
Let
\be\label{2.24}
F=\sum_{|S'|=k-1, S'\cap I=\phi} \ \prod_{j\in S'} (1-T_j) \ \prod_{j\not\in I\cup S'} T_j
\ee
satisfying by Lemma \ref{Lemma5} 
\be\label{2.25}
\Vert  F\Vert_p\leq C_p^{k-1} \Vert f\Vert_p.
\ee
Then
\be\label{2.26}
\eqref{2.23} =\Big\Vert \Big(\sum_{i\in I} \big|\Gamma_i F*\mu_i\big|^2\Big)^{\frac 12}\Big\Vert_p
\ee
where $\Gamma_i=(1-T_i) \big(\prod\limits_{\substack{j\in I\\ j\not= i}} T_j\big)$.
Assuming an inequality
\be\label{2.27}
\Big\Vert\Big(\sum_{i\in I} \big|\Gamma_i g*\mu_i\big|^2\Big)^{\frac 12}\Big\Vert_{L^p(\operatornamewithlimits\otimes\limits_{i\in I} dx_i)}
\leq b_1\Vert g\Vert_{L^p(\operatornamewithlimits\otimes\limits_{i\in I} dx_i)}
\ee
will imply that \eqref {2.23} may be bounded by $b_1.C_p^{k-1}\Vert f\Vert_p$.

Summarizing, in view of \eqref{2.22} and \eqref{2.27}, we are finally reduced to establishing inequalities
\be\label{2.28}
\Big\Vert\Big(\sum^n_{i=1} |A_0f*\mu_i|^2\Big)^{\frac 12}\Big\Vert_p \leq b_0\Vert f\Vert_p
\ee
and
\be\label{2.29}
\Big\Vert\Big(\sum^n_{i=1} \big|\Gamma_i f*\mu_i\big|^2\Big)^{\frac 12}\Big\Vert_p \leq b_1\Vert f\Vert_p \ \text { with } \ \Gamma_i= (1-T_i)\prod_{j\not= i}
T_j
\ee
for suitable $b_0=b_0(R), b_1=b_1(R)$.
From the preceding, this will permit to estimate
\be\label{2.30}
\Big\Vert\Big(\sum|\mu_i*f|^2\Big)^{\frac 12}\Big\Vert_p\leq A_p (R)\Vert f\Vert_p
\ee
\vfill\break
with
\be\label{2.31}
\begin{aligned}
A_p(R)&< C(p, K)\big(1+b_0(R)+b_1(R)\big)\\
&= C(p, \ve)(1+b_0+b_1).
\end{aligned}
\ee
Bounds on $b_0, b_1$ will be obtained in Section 4.

\section
{An auxiliary class of operators}

The key inequality is \eqref{2.29} and we will deal with it using classical techniques from martingale theory.
This will require us to introduce some additional convolution operators that are approximately stable under small translation (note that the
function $\eta(x) =(1-|x|)_+$ introduced earlier does not have this property.)

Denote
\be\label{3.1}
\vp(x) =\frac c{1+x^4} \ \text { normalized s.t. $\int^\infty_{-\infty}\vp(x)dx=1$}.
\ee
Note that
\be\label{3.2}
\vp\lesssim \vp*\vp\lesssim\vp
\ee
and
\be\label{3.3}
|\hat\vp(\lambda)|< O(e^{-c|\lambda|}) \ \text { for } \ |\lambda|\to\infty
\ee
\be\label{3.4}
|1-\hat\vp(\lambda)|<O(\lambda^2).
\ee

Let $0<t_0\ll t=R^{-\ve}$ be another parameter (to specify) and denote $L_j$ the convolution in $x_j$ by $\vp_{t_0}$,
$\vp_{t_0}(x) =\frac 1{t_0} \vp\big(\frac x{t_0}\big)$.
Hence the $\{L_j\}$ are contractions on $L^p(\mathbb R^n)$, $1\leq p\leq\infty$.

\begin{lemma}\label{Lemma7}
Assume $q\in\mathbb Z_+$ a power of 2 and $f_1, \ldots, f_n\in L^q(\mathbb R^n)$
positive functions. Then
\begin{alignat}{1}\label{3.5}
\Vert L_2\ldots L_n f_1+\cdots&+ L_1\ldots L_{n-1} f_n\Vert_q\leq\notag\\
C_q\{\Vert (L_1\ldots L_n)&(f_1+\cdots+f_n)\Vert_q 
+\Vert(L_1\ldots L_n)(f_1^2+\cdots+f_n^2)\Vert^{\frac 12}_{\frac q2}\notag\\
&+\cdots+(\Vert f_1\Vert_q^q+\cdots\Vert f_n\Vert^q_q)^{\frac 1q}\}
\end{alignat}
$$
\leq C_q\Vert f_1+\cdots +f_n\Vert_q\qquad \quad \eqno{(3.5')}
$$
\end{lemma}

\begin{proof}

The statement is obvious for $q=1$.

In general, proceed by direct calculation of
\be\label{3.6}
\int_{\mathbb R^n} \Big(\sum_j L^{(j)} f_j\Big)^q \sim \sum_{j_1\leq j_2\leq \cdots\leq j_q}\, \int (L^{(j_1)} f_{j_1})\cdots (L^{(j_q)} f_{j_q})
\ee
denoting $L_1\ldots L_{\hat j}\ldots L_n=L^{(j)}$.

Using H\"older's inequality, the contribution of $j_1=j_2$ in \eqref{3.6} is bounded by
$$
\int\Big[\sum_j(L^{(j)} f_j)^2\Big]\, \Big[\sum_j L^{(j)} f_j\Big]^{q-2} \leq\Big\Vert\sum_j L^{(j)} f_j^2\Big\Vert_{\frac q2}\Big\Vert\sum_j L^{(j)} f_j\Big\Vert_q^{q-2}
$$
reducing $q$ to $\frac q2$.
For the $j_1<j_2$ contribution, proceed as follows.

We can assume $j_1 =1$ and rewrite the integral in the r.h.s of \eqref{3.6} as
\be\label{3.7}
\int_{\mathbb R^n} g_1(L_1g_2)\cdots(L_1 g_q)
\ee
with $g_1= L^{(1)} f_1$ etc.
Integration in $x_1$ gives
\be\label{3.8}
\int g_1(x_1) g_2(x_1-y_2)\cdots g_q(x_1-y_q) \vp_{t_0} (y_2)\cdots\vp_{t_0} (y_q) dx_1 dy_2\ldots dy_q.
\ee
Perform a translation $x_1\mapsto x_1+\tau$ with $|\tau|<t_0$ and use the property that $\vp_{t_0} (y+\tau)\sim \vp_{t_0}(y)$ for $|\tau|\leq t_0$. This gives that
$$
\eqref{3.8} \sim\int g_1 (x_1+\tau) g_2(x_1-y_2)\cdots g_q(x_1-y_q) \vp_{t_0}(y_2)\cdots \vp_{t_0} (y_q) dx_1 dy_2\ldots dy_q
$$
and averaging over $|\tau|\leq t_0$
$$
\eqref{3.7} \lesssim \int_{\mathbb R^n} (L_1g_1)(L_1g_2)\cdots (L_1 g_q).
$$
Thus the $j_1<j_2$ contribution in \eqref{3.6} may be estimated by
$$
\begin{aligned}
&\int(L_1\ldots L_n) \Big(\sum f_j\Big) \Big(\sum L^{(j)} f_j\Big)^{q-1}\lesssim\\
&\Vert(L_1\ldots L_n) \Big(\sum f_j\Big)\Big\Vert_q \, \Big\Vert\sum L^{(j)} f_j\Big\Vert^{q-1}_q
\end{aligned}
$$
proving the Lemma.
\end{proof}

Recall the definition of $\mu_i=\partial_{x_i} (1_B*H_{\frac 1R})$.
Using Lemma \ref{Lemma7}, we prove
\begin{lemma}\label{Lemma8}
For $q$ as above and $f_1, \ldots, f_n \in L^q(\mathbb R^n)$
\be\label{3.9}
\Big\Vert\Big[ \sum^n_{i=1} |(\partial_{x_i} 1_B)*L^{(i)} f_i\Big|^2\Big]^{\frac 12}\Big\Vert_g \leq c_q R^{8\ve}\Big\Vert\Big(\sum^n_{i=1} |f_i|^2\Big)^{\frac 12}\Big\Vert_q.
\ee
\end{lemma}

\begin{proof}

Note that
$$
|\partial_i 1_B|\leq (\delta_{\frac 12}+\delta_{-\frac 12})(x_i) . 1_{B^{(i)}}
$$
with $B^{(i)} =\prod_{j\not= i} \big[-\frac 12, \frac 12\big]\subset\mathbb R^{n-1}$.

Thus
$$
\sum_i|\partial_i 1_B*L^{(i)} f_i|^2 \leq\sum_i L^{(i)} \tau_i \big(|f_i|^2 *1_{B^{(i)}}\big)
$$
with $\tau_i$ the shift $x_i\mapsto x_i\pm \frac 12$ and we evaluate the $L^{\frac q2}$-norm applying \eqref{3.5}.
This gives the expressions
\be\label{3.10}
\Big\Vert\sum_i(L_1\ldots L_n) [\tau_i(|f_i|^2*1_{B^{(i)}})\Big]^{2^s}\Big\Vert_{\frac q{2^{s+1}}}
\ee
with $1\leq 2^s \leq\frac q2$ and
\be\label{3.11}
\eqref{3.10} \leq\Big\Vert\sum_i L_1\ldots L_n\tau_i(*1_{B^{(i)}}) (|f_i|^{2^{s+1}})\Big\Vert_{\frac q{2^{s+1}}}.
\ee

Next, observe that since $\vp_{t_0} (x+\tau)\leq C t_0^{-4} \vp_{t_0}(x)$, we have
\be\label{3.12}
L_i\tau_i< CR^{4\ve} L_i < CR^{8\ve} L_i(*1_{[-\frac 12, \frac 12]})
\ee
with convolution in the $x_i$-variable.
Therefore
$$
\begin{aligned}
\eqref{3.11} &\leq CR^{8\ve}\Big\Vert\sum_i (L_1\ldots L_n)(|f_i|^{2^{s+1}} *1_B)\Big\Vert_{\frac q{2^{s+1}}}\\
&\leq CR^{8\ve} \Big\Vert\Big(\sum|f_i|^2\Big)^{\frac 12}\Big\Vert_q
\end{aligned}
$$
proving Lemma \ref{Lemma8}.
\end{proof}

As a corollary

\begin{lemma}\label{Lemma9}
\be\label {3.13}
\Big\Vert\Big(\sum^n_{i=1} |\mu_i*L^{(i)} f_i|^2\Big)^{\frac 12}\Big\Vert_q \leq C_q R^{8\ve}\Big\Vert\Big(\sum |f_i|^2\Big)^{\frac 12}\Big\Vert_q
\ee
\end{lemma}

\section
{Completion of the proof}

Return to inequalities \eqref{2.28}, (2.29).
We may assume $p$ a power of 2.
Set
\be\label{4.1}
t_0=R^{-3\ve}
\ee
and let $\{L_j\}$ and $\{L^{(i)}\}$ be the operators introduced in Section 3.

Consider first (2.28) and estimate
\begin{alignat}{1} \label{4.2} 
&\Big\Vert\Big(\sum_i |A_0f*\mu_i|^2\Big)^{\frac 12}\Big\Vert_p\leq\notag\\
&\Big\Vert\Big(\sum_i |A_of*H_{t_0} *\mu_i|^2\Big)^{\frac 12} \Big\Vert_p
\\
&+ \Big\Vert \Big(\sum_i|A_0(1-H_{t_0}) f* \mu_i|^2\Big)^{\frac 12}\Big\Vert_p.
\end{alignat}

Since $H_{t_0}*\mu_i=\partial_i H_{t_0} *H_{\frac 1R}*1_B$,
\be
\label{4.4}
\eqref{4.2} \leq\Big\Vert\Big( \sum_i |\partial_i H_{t_0} * 1_B*f|^2\Big)^{\frac 12} \Big\Vert_p< C t_0^{-1} \Vert f\Vert_p
\ee using interpolation between $p=2$, $p=\infty$ and recalling (1.25)-(1.28).

It follows from the definition of $A_p$ in \eqref{2.30} that
\be\label{4.5}
(4.3)\leq A_p\Vert A_0(1-H_{t_0}) f\Vert_p
\ee
and we estimate $\Vert A_0(1-H_{t_0})\Vert_p$ by interpolation.

Obviously
$$
\Vert A_0(1-H_{t_0})\Vert_\infty \lesssim \Vert A_0\Vert_\infty (1+\Vert H_{t_0}\Vert_\infty)=2
$$
while for $p=2$, we need to bound the multiplier
\be\label{4.6}
\prod^n_{i=1} \hat\eta (t\xi_i)( 1- e^{-t_0^2|\xi|^2}).
\ee
Since $|\hat\eta (\lambda)|< C\lambda^{-2}$, certainly
$$
\eqref{4.6} < CR^{2\ve} (\max |\xi_i|)^{-2}< CR^{-\ve}
$$
unless $\max |\xi_i|< R^{2\ve}$.

Also $|\hat\eta (\lambda)|< e^{-C\lambda^2} $ for $|\lambda|<1$ and hence
$$
\prod^n_{i=1} |\hat\eta(t\xi_i)|< e^{-ct^2 (\sum_{|\xi_i|<R^\ve} \xi_i^2) -c|I_1|}
$$
with $I_1=\{i\leq n; |\xi_i|\geq R^\ve\}$.
Thus also $\eqref{4.6}<R^{-\ve}$ unless
$$
\sum_{|\xi_i|< R^\ve}\xi_i^2 <CR^{2\ve}\log R \text { and } \ |I_1|< C\log R
$$
and we can assume
$$
|\xi|^2 < CR^{2\ve}\log R+|I_1|R^{4\ve} < CR^{4\ve}\log R.
$$
But then
$$
1-e^{-t_0^2|\xi|^2} \lesssim t_0^2 R^{4\ve} \log R<R^{-\ve}
$$
by \eqref{4.1}. This proves that
$$
\eqref{4.6} < CR^{-\ve}
$$
and consequently
\be\label{4.7}
\Vert A_0(1-H_{t_0})\Vert_2< CR^{-\ve}.
\ee
Interpolation with $p=\infty$ gives
\be\label{4.8}
\Vert A_0(1-H_{t_o})\Vert_p< CR^{-\frac {2\ve}p}
\ee
and
\be\label{4.9}
\eqref{4.5}\leq C A_p R^{-\frac {2\ve}p}\Vert f\Vert_p.
\ee
From \eqref{4.4}, \eqref{4.9}, we find
\be\label{4.10}
b_0<CR^{3\ve}+ CA_p R^{-\frac {2\ve}p}.
\ee
Consider next \eqref{2.29} and estimate
\begin{alignat}{1}\label{4.11}
\Big\Vert \Big(\sum\big|\Gamma_i f* \mu_i\big|^2\Big)^{\frac 12} \Big\Vert_p &\leq\Big\Vert\Big(\sum_i 
\big|\Gamma_i L^{(i)} f*\mu_i\big|^2\Big)^{\frac 12}
\Big\Vert_p  \\
&+\Big\Vert\Big(\sum_i \big|\Gamma_i(1-L^{(i)})f*\mu_i\big|^2\Big)^{\frac 12}\Big\Vert_p.
\label{4.12}
\end{alignat}
Application of Lemma \ref{Lemma9} with $f_i =\Gamma_i f$ gives
\be
\label{4.13}
\eqref{4.11} \leq CR^{8\ve}\Big\Vert \Big(\sum\big|\Gamma_i f\big|^2\Big)^{\frac 12}
\Big\Vert_p.
\ee
Note that application of Lemma \ref{Lemma5} to a subset of the variables implies that for any $I\subset\{1, \ldots, n\}$,
$$
\Big\Vert\sum_{i\in I} \Gamma_if\Big\Vert_p \leq C_p\Vert f\Vert_p\qquad (1<p<\infty).
$$
Hence
\be\label{4.14}
\Big\Vert\Big(\sum \big |\Gamma_i f\big|^2\Big)^{\frac 12}\Big\Vert_p \leq C_p\Vert f\Vert_p
\ee
and
\be\label{4.15}
\eqref{4.11} \leq C_pR^{8\ve}.
\ee

Evaluate
\be\label{4.16}
\eqref{4.12} \leq \mathbb E_\ve\Big[\Big\Vert\Big(\sum_i |F_\ve*\mu_i|^2\Big)^{\frac 12}\Big\Vert_p\Big]
\ee
where $\ve\in \{1, -1\}^n$ and $F_\ve=\sum_i\ve_i\Gamma_i(1-L^{(i)}) f$.

From (2.30)
\begin{alignat}{1}\label{4.17} 
\eqref{4.16} &\leq A_p\mathbb E_\ve [\Vert F_\ve\Vert_p]\notag\\
&\leq C_p A_p \Big\Vert\Big(\sum_i\big|\Gamma_i(1-L^{(i)}) f
\big|^2\Big)^{\frac 12}\Big\Vert_p.
\end{alignat}

By \eqref{4.14} and (3.5'), it follows that
\begin{alignat}{1}
\Big\Vert\Big(\sum\big|\Gamma_i(1-L^{(i)}) f\big|^2\Big)^{\frac 12}\big\Vert_p &\leq \big\Vert\Big(\sum\big|\Gamma_i f\big|^2\Big)^{\frac 12}\Big\Vert_p
+\Big\Vert\sum L^{(i)} |\Gamma_i f|^2 \Big\Vert^{\frac 12}_{\frac p2}\\ \notag
&\leq C_p \Vert f\Vert_p
\end{alignat}
and we interpolate again with an $L^2$-bound.
The latter is obtained by bounding the multiplier
\begin{alignat}{1}\label{4.19}
&\sum_i|\hat\Gamma_i(\xi)|^2 \ |1- \widehat{L^{(i)}} (\xi)|^2= \notag\\
& \sum_i|1-\hat\eta (t\xi_i)|^2 \prod_{j\not=i}|\hat\eta (t\xi_j)|^2 \ | 1-\prod_{j\not= i}\hat\vp (t_0\xi_j)|^2\leq\notag\\
&\max_i\prod_{j\not= i} \hat\eta(t\xi_j) |1-\prod_{j\not=i}\hat\vp (t_0\xi_j)|.
\end{alignat}
Since $|1-\hat\vp(\lambda)|< C\lambda^2$ by \eqref{3.4}, \eqref{4.19} may be estimated as \eqref{4.6} and
hence
\be\label{4.20}
\Big\Vert \Big(\sum\big|\Gamma_i(1-L^{(i)})f\big|^2\Big)^{\frac 12}\Big\Vert_2 \leq CR^{-\ve} \Vert f\Vert_2
\ee
\be\label{4.21}
\Big\Vert\Big(\sum\big|\Gamma_i(1-L^{(i)} ) f\big|^2\Big)^{\frac 12}\Big\Vert_p\leq CR^{-\ve\frac 2p}\Vert f\Vert_p.
\ee
From \eqref{4.15}, \eqref{4.21}, we may take
\be\label{4.22}
b_1< C_p R^{8\ve}+ C_pA_pR^{-\frac 2p\ve}.
\ee
Finally, from \eqref{2.31}, \eqref{4.10}, \eqref{4.22}, we deduce
\begin{alignat}{1}\label{4.23}
A_p(R)&< C(p, \ve) (R^{8\ve}+A_p R^{-\frac 2p\ve})\notag\\
A_p(R)&<C_p (\ve) R^{8\ve}.
\end{alignat}
This completes the proof of Lemma \ref{Lemma3} and the Theorem.


\begin{thebibliography}
{xxxxxxx}

\bibitem
[Al]{Al} J.M.~Aldaz, \emph {The weak type $(1, 1)$ bounds for the maximal function associated to cubes grow to infinity with the dimension},
Annals of Math (2), {\bf 173} (2011), n2, 1013--1023.

\bibitem
[Au]{Au} G.~Aubrun, \emph {Maximal inequality for high dimensional cubes}, arXiv:0902.4305.

\bibitem
[B1]{B1} J.~Bourgain, \emph{On high dimensional maximal functions associated to convex bodies}, Amer.J.~Math., {\bf 108} (1986), 1467--1476.

\bibitem[B2]{B2} J.~Bourgain, \emph{On the $L^p$-bounds for maximal functions associated to convex bodies in $R^n$}, Israel J.~Math., {\bf 54} (1986), 257--265.

\bibitem [C]{C} A.~Carbery, \emph{ Radial Fourier multipliers and associated maximal functions}, Recent Progress in Fourier Analysis, Noth-Holland Math. Studies,
{\bf 111} (1985), 49--56.

\bibitem[M]{M} D.~M\"uller, \emph{A geometric bound for maximal functions associated to convex bodies}, Pacific J.~Math., Vol. {\bf 142}, N2 (1990).

\bibitem [N-T]{N-T} A.~Naor, T.~Tao, \emph{Random martingales and localization of maximal inequalities}, JFA 259 (2010), no 3, 731--779.

\bibitem[P]{P} G.~Pisier, \emph{Holomorphic semi-groups and the geometry of Banach spaces}, Annals of Math., {\bf 115} (1982), 375--392.

\bibitem
[St1]{St1} E.~Stein, \emph{The development of square functions in the work of A.~Zygmund}, Bull. Amer. Math. Soc., {\bf 7} (1982), 359--376.

\bibitem
[St2]{St2} E.~Stein, 
\emph {Three variations on the theme of maximal functions}, Recent Progress in Fourier Analysis, North-Holland Math. Studies {\bf 111} (1985),
49--56.

\bibitem[S-S]{S-S} E.M.~Stein, J.O.~Str\"omberg, \emph{Behavior of maximal functions in $\mathbb R^n$ for large $n$}, Ark. F. Mat., {\bf 21} (1983), 259--269.


\end{thebibliography}
\end{document}